\documentclass{amsart}

\newtheorem{theorem}{Theorem}[section]
\newtheorem{lemma}[theorem]{Lemma}
\newtheorem{corollary}[theorem]{Corollary}

\theoremstyle{definition}
\newtheorem{defi}[theorem]{Definition}
\newtheorem{example}[theorem]{Example}

\newenvironment{prf} {{\bf Proof.}}{\hfill $\Box$}

\theoremstyle{remark}

\newcommand{\be}{\begin{equation}}
\newcommand{\ben}{\end{equation}}

\numberwithin{equation}{section}
\newcommand{\Lp}{\ell^p}
\newcommand{\Ltwo}{\ell^2}

\begin{document}

\title[Poincar\'e  and Plancherel-Polya inequalities in analysis  on graphs]{Poincar\'e  and Plancherel-Polya inequalities in harmonic analysis on weighted combinatorial graphs}

\author{Hartmut F\"{u}hr}
\address{Lehrstuhl A f\"ur Mathematik, RWTH Aachen, D-52056 Aachen, Germany}
\email{fuehr@matha.rwth-aachen.de}

\author{Isaac Z. Pesenson} \footnote{The author was supported in
part by the National Geospatial-Intelligence Agency University
Research Initiative (NURI), grant HM1582-08-1-0019. }
\address{Department of Mathematics, Temple University, Philadelphia, PA 19122, USA}
\email{pesenson@temple.edu}


\begin{abstract}We prove Poincar\'e and Plancherel-Polya 
inequalities for weighted $\ell^p$ -spaces on weighted graphs in which the constants are explicitly expressed in terms of some geometric characteristics of a graph. We use Poincar\'e type inequality to obtain  some new relations between geometric and spectral properties of the combinatorial Laplace operator. Several well known graphs are considered to demonstrate that our results are reasonably sharp. 

The Plancherel-Polya inequalities allow for application of the frame algorithm as a method for reconstruction of Paley-Wiener functions on weighted graphs from a set of samples. The  results are illustrated by developing Shannon-type sampling  in the case of a line graph.

Our work has potential applications to data mining and learning theory on graphs.
\end{abstract}

\maketitle

\section{Introduction and main results}

\label{sect:main}

Let $G$ denote an undirected weighted graph, with vertices $V(G)$ and weight function $w: V(G) \times V(G) \to \mathbb{R}_0^+$. $w$ is symmetric, i.e., $w(u,v) = w(v,u)$, and $w(u,u)=0$ for all $u,v \in V(G)$. The edges of the graph are the pairs $(u,v)$ with $w(u,v) \not= 0$. We fix a strictly positive weight $\nu: V(G) \to \mathbb{R}^+$ and let $\Lp_\nu(G)$ denote the space of  functions $f : V(G) \to \mathbb{C}$  satisfying
\[
 \| f \|_{p,\nu} = \left( \sum_{v \in V(G)} |f(v)|^p \nu(v) \right)^{1/p}~,~1 \le p < \infty~.
\] 
For $\nu \equiv 1$ we write $\Lp (G) := \Lp_\nu(G)$. 
Our arguments in Sections 1 and 2 work for arbitrary choices of $\nu$. In all other sections we consider only the spaces $\Ltwo(G)$. 

We intend to prove Poincar\'e-type estimates involving the  {\em weighted gradient $p$-norm} of a function $f$ on $V(G)$, defined as 
\[
 \| \nabla_w f \|_{p} = \left( \sum_{u,v \in V(G)} |f(u) - f(v)|^p w(u,v) \right)^{1/p}~,
\] for $1 \le p < \infty$.  It will not be necessary to study the analogous notion for $p=\infty$; the equality $\| \nabla_w f \|_{\infty} = \| \nabla f \|_{\infty}$ implies that the weighted case does not add anything new here.  The term  Poincar\'e-type inequality is used for estimates in which the norm of a function is estimated through the norm of its gradient.  Poincar\'e-type inequalities on graphs were considered in  a number  of papers (see \cite{CG}, \cite{CK} and references there). However, our goals and methods are rather different from objectives and approaches of other mathematicians.

Before we describe the main results in more detail, let us give a brief overview of the paper: Section 2 contains the proofs of the main results, Theorems \ref{thm:main} and \ref{thm:main_2}. The role of the remaining sections is to explore the consequences of these results: In Section 3, we relate the constants entering our estimates to already established structural constants for graph laplacians, such as Dirichlet eigenvalues and isoperimetric constants, and study explicit examples showing that our constants are often close to optimal. Sections 4 and 5 study sampling application, with Section 4 focussing on frame-theoretic. Section 5 evaluates the graph-theoretic estimates for a setting where the optimal answers are known (i.e., Shannon sampling on the integers), and shows that also in this case, our constants are close to optimal.   

The central estimates of this paper will be derived from a suitable partition $ \mathcal{S}  = (S_m)_{m =0,\ldots n}$ of $V(G)$, and certain quantities describing the weights associated to neighboring $S_j$. Given any subset $A \subset V(G)$ and $v \in V(G)$, we let $w_A(v) = \sum_{u \in A} w(u,v)$. We note that $w_A(v) = 0$ iff there is no edge connecting $v$ and some element of $A$.

Given a finite partition $\mathcal{S} = (S_m)_{m=0,\ldots,n}$ of $V(G)$, we let
\[
 D_{m}  = D_m(\mathcal{S}) = \sup_{v \in S_m} \frac{w_{S_{m+1}}(v)}{\nu(v)} 
\]and 
\[
 K_m = K_m (\mathcal{S}) = \inf_{v \in S_{m+1}} \frac{w_{S_m(v)}}{\nu(v)}~.
\] 
We let $n(\mathcal{S}) = n$, the {\em length of } $\mathcal{S}$. The set $S_0$ is called {\em inital set} of the partition $\mathcal{S}$, it is of primary importance for the following results. The partition is called {\em admissible} if $K_m>0$ holds for all $m=0,\ldots,n-1$. A necessary condition for admissibility is that every $v \in S_m$ has a common edge with some $u \in S_{m-1}$. We will write $f_m$ for the restriction of $f \in \Lp_\nu(G)$ to $S_m$. 

With these notations, we can now formulate the central estimate of this paper.  The definition of the constants relies on the convention that a product over an empty index domain equals one by definition, and likewise, the sum over an empty index domain equals zero. This applies to all products and sums for which the lower index bound exceeds the upper bound. 
\begin{theorem} \label{thm:main}
Suppose that $\mathcal{S}$ is a partition of $V(G)$, and $n = n(\mathcal{S})$. Assume that $K_m(\mathcal{S}) >0$ for all $m=0,\ldots, n-1$. Then, for all $1 < p < \infty, \>\>1 < q < \infty$ with $1/p+1/q=1$, and $f \in \Lp_\nu(G)$, we have for $f_{0}=f|_{S_{0}}$
\[
 \| f \|_{p,\nu} \le \left( \sum_{m=0}^{n} \prod_{j=0}^{m-1} \frac{D_j}{K_j} \right)^{1/p} \| f_0 \|_{p,\nu} +
 $$
 $$
 \left( \sum_{m=1}^n \left( \sum_{k=1}^m \frac{1}{K_{k-1}^{q/p}} \left( \prod_{i=k}^{m-1} \frac{D_i}{K_i} \right)^{q/p} \right)^{p/q} \right)^{1/p}  \| \nabla_w f \|_{p}~.
\]
For $f \in \ell^1_\nu(G)$, we have 
\[
 \| f \|_{1,\nu} \le \left( \sum_{m=0}^{n} \prod_{j=0}^{m-1} \frac{D_j}{K_j} \right) \| f_0 \|_{1,\nu} + \max_{k=1,\ldots,n} 
\left( \frac{1}{K_{k-1}} \sum_{m=k}^n \prod_{i=k}^{m-1} \frac{D_i}{K_i} \right) \| \nabla_w f \|_{1}~. 
\]
\end{theorem}

The following corollary notes how these estimates help characterize sampling sets: 
\begin{corollary} \label{cor:samp_1}
 Let $X \subset \Lp_\nu (G)$ denote a subspace such that for all $f \in X$, the inequality $\| \nabla_w f \|_{p,\nu} \le \epsilon \| f \|_{p,\nu}$ holds ($1 \le p < \infty$). 
Let $\mathcal{S}$ be a partition of $V(G)$, and let $n =n(\mathcal{S})$. Assume that $K_m(\mathcal{S}) >0$ for all $m=0,\ldots,n-1$, and such that 
\[
\delta_{\mathcal{S},p} = \left\{ \begin{array}{cc} 
 \max_{k=1,\ldots,n} 
\left( \frac{1}{K_{k-1}} \sum_{m=k}^n \prod_{i=k}^{m-1} \frac{D_i}{K_i} \right), & p= 1, \\
\left( \sum_{m=1}^n \left( \sum_{k=1}^m \frac{1}{K_{k-1}^{q/p}} \left( \prod_{i=k}^{m-1} \frac{D_i}{K_i} \right)^{q/p} \right)^{p/q} \right)^{1/p}, & 1 < p < \infty , \>\frac{1}{p}+\frac{1}{q} = 1, \end{array} \right.
\] fulfills $\epsilon \delta_{\mathcal{S},p} < 1$. Define 
\[
 a_{\mathcal{S},p} =  \left( \sum_{m=0}^{n} \prod_{j=0}^{m-1} \frac{D_j}{K_j} \right)^{1/p}
\]
Then the following Plancherel-Polya-type inequalities hold  for all $f \in X$: 
\begin{equation} \label{eqn:cor_norm_equiv}
 \frac{1-\epsilon \delta_{\mathcal{S},p}}{a_{\mathcal{S},p}} \| f \|_{p,\nu} \le \| f |_{S_0} \|_{p,\nu} \le \| f \|_{p,\nu}.
\end{equation}
\end{corollary}

We call (\ref{eqn:cor_norm_equiv}) Plancherel-Polya-type inequalities since they relate the norm of a function to its norm on a subset. In the classical case similar inequalities were established by Plancherel and Polya for functions whose Fourier transform has compact support. Such inequalities are also known as frame inequalities. 

It is also desirable to obtain nontrivial upper bounds in the sampling estimate. For this purpose, we consider a finite sequence $\widehat{\mathcal{S}} = (S_m)_{m=0,\ldots,n}$ of disjoint subsets of $V(G)$, where this time $\bigcup S_m \not= V(G)$ is admitted. Again, we let $n(\widehat{\mathcal{S}}) = n$, and call $S_0$ the {\em initial set}. 

For $0 \le m < n$ let
\[
 \widehat{K}_m(\widehat{\mathcal{S}}) = \widehat{K}_m = \inf_{v \in S_{m}} \frac{w_{S_{m+1}}(v)}{\nu(v)}~,
\] 
as well as
\[
 \widehat{D}_m (\widehat{\mathcal{S}}) = \widehat{D}_m = \sup_{v \in S_{m+1}} \frac{w_{S_{m}}(v)}{\nu(v)}~.
\]

\begin{theorem}
 \label{thm:main_2}
Let $\widehat{\mathcal{S}}$ be a finite sequence of disjoint subsets of $V(G)$, and $n = n(\widehat{\mathcal{S}})$. Assume that $\widehat{K}_m(\widehat{\mathcal{S}}) ,\widehat{D}_m(\widehat{\mathcal{S}}) >0$, for all $ m =0,\ldots,n-1$. Then, for all $1 < p < \infty, \>\>1 < q < \infty,$ with $1/p+1/q=1$, and $f \in \Lp_\nu(G)$, we have the inequality
 \[
 \| f \|_{p,\nu}+\left( \sum_{m=1}^n \left( \sum_{k=0}^{m-1}  \frac{1}{\widehat{K}_{k}^{q/p}} \left( \prod_{i=k}^{m-1} \frac{\widehat{K}_i}{\widehat{D}_i} \right)^{q/p} \right)^{p/q} \right)^{1/p}  \| \nabla_w f \|_{p}\ge
 $$
 $$
  \left( \sum_{m=0}^{n} \prod_{j=0}^{m-1} \frac{\widehat{K}_j}{\widehat{D}_j} \right)^{1/p} \| f_0 \|_{p,\nu}.
\]
For $f \in \ell^1_\nu(G)$, we have 
\[
 \| f \|_{1,\nu} +  \max_{k=0,\ldots,n-1} \frac{1}{\widehat{K}_k} \sum_{m=k}^{n}  \left( \prod_{i=k}^{m-1} \frac{\widehat{K}_i}{\widehat{D}_i} \right)  \| \nabla_w f \|_{1} \ge  \left( \sum_{m=0}^{n} \prod_{j=0}^{m-1} \frac{\widehat{K}_j}{\widehat{D}_j} \right) \| f_0 \|_{1,\nu}~.
\]
\end{theorem}

The theorem  allows to sharpen Corollary \ref{cor:samp_1} some more:
\begin{corollary} \label{cor:samp_2}
 Let $X \subset \Lp_\nu(G)$ denote a subspace such that for all $f \in X$, the inequality $\| \nabla f \|_{p,\nu}\le \epsilon \| f \|_{p,\nu}$ holds. Let $\widehat{\mathcal{S}}$ be finite sequence of disjoint subsets  of $V(G)$, and let $n =n(\widehat{\mathcal{S}})$. Let the constants
$\delta_{\mathcal{S},p},a_{\mathcal{S},p}$ be defined as in Corollary \ref{cor:samp_1}, associated to a suitable disjoint covering $\mathcal{S}$, possibly different from $\widehat{\mathcal{S}}$, but with the same initial set $S_0$. Assume that the constants
\[
\hat{\delta}_{\widehat{\mathcal{S}},p} = \left\{ \begin{array}{cc} 
  \max_{k=0,\ldots,n} \frac{1}{\widehat{K}_k} \sum_{m=k}^{n}  \left( \prod_{i=k}^{m-1} \frac{\widehat{K}_i}{\widehat{D}_i} \right), & p= 1, \\
\left( \sum_{m=1}^{n} \left( \sum_{k=0}^{m-1} \frac{1}{\widehat{K}_{k}^{q/p}} \left( \prod_{i=k}^{m-1} \frac{\widehat{K}_i}{\widehat{D}_i} \right)^{q/p} \right)^{p/q} \right)^{1/p} , & 1 < p < \infty, \>\> \frac{1}{p}+\frac{1}{q} = 1  \end{array} \right.
\] and 
\[
 \hat{a}_{\widehat{\mathcal{S}},p} =   \left( \sum_{m=0}^{n} \prod_{j=0}^{m-1} \frac{\widehat{K}_j}{\widehat{D}_j} \right)^{1/p},
\] are well-defined, i.e., $\widehat{K}_m,\widehat{D}_m>0$, for $0 \le m < n$. 
Then, if $\epsilon \delta_{\mathcal{S},p}< 1$,  the following Plancherel-Polya-type  equivalence holds for all $f \in X$: 
\begin{equation} \label{eqn:cor_norm_equiv_sharp}
 \frac{1-\epsilon \delta_{\mathcal{S},p}}{a_{\mathcal{S},p}} \| f \|_{p,\nu}\le \| f |_S \|_{p,\nu}\le \frac{1+\epsilon \hat{\delta}_{\widehat{\mathcal{S}},p}}{\hat{a}_{\widehat{\mathcal{S}},p}} \| f \|_{p,\nu}.
\end{equation}
\end{corollary}

Theorems \ref{thm:main}, \ref{thm:main_2} and their corollaries may be seen as a generalization and sharpening of \cite[Theorem 2.1]{PePe2010}; we will comment on the improvement in terms of sharpness when we discuss Shannon sampling on the integers in Section \ref{sect:shannon_sampling}. 
As the discussion in that section will show, it may occur that $a_{\mathcal{S},p} \approx \hat{a}_{\widehat{\mathcal{S}},p}$ (e.g., for a large variety of sampling sets $S$, the two only differ up to universal multiplicative constants). In this setting, the {\em tightness} of the estimate, i.e., the quotient of upper and lower bound, will be proportional to the quotient $\frac{\displaystyle 1+\epsilon \hat{\delta}_{S,p}}{\displaystyle 1-\epsilon \delta_{S,p}}$. In this case the constants $a_{S,p}$ and $\hat{a}_{S,p}$ assume the role of {\em normalization constants} that do not significantly affect the tightness of the estimate.

We also have two other consequences of our main estimates which show their direct relevance to Poincare-type inequalities.

If a function $f \in \ell^p(G)$ is supported on $V\setminus S_{0}$
\begin{equation}\label{restriction_S_{0}}
f|_{S_{0}}=0
\end{equation}
then $\|f\|_{p,\nu}=\|f|_{V\setminus S_{0}}\|_{p,\nu}$ and $\|f|_{S_{0}}\|_{p,\nu}=0$. In this situation  Theorem \ref{thm:main} implies the following Corollary.

\begin{corollary}\label{col_1}

Suppose that $\mathcal{S}$ is an admissible partition of $V(G)$, and $n = n(\mathcal{S})$. Then, for all $1 \leq p < \infty$, and $f \in \Lp_\nu(G)$ such that $f|_{S_{0}}=0$
we have the following Poincar\'e-type inequality
\[
 \| f |_{V\setminus S_{0}}\|_{p,\nu} \le  \delta_{\mathcal{S},p} \| \nabla_{w} f \|_{p}~,
\]
and for $f \in \ell^\infty(G)$, we have
\[
 \| f|_{V\setminus S_{0}}\|_\infty \le  n \| \nabla_{w} f \|_{\infty}~. 
\]

\end{corollary}

If a function $f \in \ell^p(G)$ is supported on $S_{0}$
\begin{equation}
f|_{V\setminus S_{0}}=0
\end{equation}
then $\|f\|_{p,\nu}=\|f|_{S_{0}}\|_{p,\nu}$ and $\|f|_{V\setminus S_{0}}\|_{p,\nu}=0$. In this situation  Theorem \ref{thm:main_2} implies the following Corollary.
\begin{corollary}
 \label{col_2}
 
Let $\widehat{\mathcal{S}}$ be a finite sequence of disjoint subsets of $V(G)$, and $n = n(\widehat{\mathcal{S}})$. Assume that $\widehat{K}_m(\widehat{\mathcal{S}}) ,\widehat{D}_m(\widehat{\mathcal{S}}) >0$, for all $ m =0,\ldots,n-1$. Then, for all $1 \le p < \infty,$ and $f \in \ell^p(G)$, such that 
$
f|_{V\setminus S_{0}}=0
$
we have the following Poincar\'e-type inequality 

 \[
 \| f|_{S_{0}}\|_{p,\nu} \leq   \frac{\widehat{\delta}_{\mathcal{S},p} }   { \widehat {a}_{\mathcal{S},p}-1  }\| \nabla_{w} f \|_{p} ~.
\]
\end{corollary}

Clearly, the crux of the approach is the choice of the partition $\mathcal{S}$ (and $\widehat{\mathcal{S}}$). Note that in all sampling estimates,  the sampling set $S_0$ is of primary importance, whereas $S_1,\ldots,S_n$ are of a strictly auxiliary nature; they should be chosen to ensure small constants in the Poincar\'e estimate. The question of choosing the partition $\mathcal{S}$ (and $\widehat{\mathcal{S}}$), given the sampling set $S_0$, in a way that guarantees good control over constants, remains an interesting challenge. For the unweighted case, there is a natural approach via repeated closure operations, employed for similar purposes in \cite{PePe2010}: Given $S \subset V(G)$, we let $cl(S) = S \cup \{ v \in V(G) : \exists u \in S \mbox{ with } u \sim v \}$, and $b(S) = cl(S) \setminus S$. Define iteratively $cl^{m+1}(S) = cl(cl^m(S))$. Thus, if we pick $S$ with $cl^{n-1}(S) \subseteq V(G) = cl^n(S)$, then letting $S_m = b(cl^{m-1}(S))$, for $m \ge 1$, and $S_0 = S$, yields a partition $\mathcal{S} = \{ S_0,\ldots,S_{n} \}$ with the property that each element of $S_m$ is connected to at least one element of $S_{m-1}$. In the case where the nonzero values of the weight $w$ have a nontrivial lower bound $\kappa>0$, this implies $K_m \ge \kappa$; this applies in particular to all finite and/or unweighted graphs. 

In the general weighted case this approach may not work, and it also might not be advisable even when a nontrivial lower bound $\kappa>0$ is avalaible. E.g., consider a setting where $w(u,v)>0$ for all $u \not= v$, but with small values for most pairs $(u,v)$. Here a single closure operation will always yield the whole set, i.e. one has $V(G) = S_0 \cup b(S_0)$, but the resulting constants for this partition may be far from optimal. 

In  the unweighted $\Ltwo(G)$-space  where $\nu \equiv 1$ the weighted Laplace operator $L_w: \Ltwo(G) \to \Ltwo(G)$  is introduced  via
\begin{equation}\label{L}
 (L_w f)(v) = \sum_{u \in V(G)} (f(v)-f(u)) w(v,u)~.
\end{equation}
The graph Laplacian is a well-studied object; it is known to be a positive-semidefinite self-adjoint \textit{bounded} operator as long as the degrees 
$$
d(u) = \sum_{u \in V(G)} w(u,v), 
$$
of all vertices are uniformly bounded.

\textit{We do not  assume that degrees of all vertices are uniformly bounded. Without going to details, we just assume a situation in which graph Laplacian $L_{w}$ is self-adjoint in $\ell^{2}(G)$,   thus possesses a spectral decomposition. }

The associated Paley-Wiener spaces are then given by 
\begin{defi}\label{PW}
 $PW_{\omega} (L_w)\subset \Ltwo(G)$ denote the image space of the projection operator ${\bf 1}_{[0,\sqrt{\omega}]}(L_{w})$ (to be understood in the sense of Borel functional calculus). 
 \end{defi}

It is important to realize, that the quantities $D_{m},\>K_{m},\>$, as well as $\widehat{D}_m, \>\widehat{K}_m$    provide an important characterization  of diffusion geometry  of a graph. Our inequalities reveal the role they are playing in analysis on the graph.  In Section 3 we incorporate the Laplace operator and some of its spectral characteristics. In this way we obtain rather interesting relations between geometry of a graph and spectral properties of Laplacian. 
These relations are presented in the form of frame inequalities (\ref{frame-ineq}).  In turn, we use the frame inequalities to describe an efficient way to reconstruct   Paley-Wiener functions from their values on sampling sets (Theorem \ref{recon}).

The Plancherel-Polya inequalities which are proved in Section 2 allow for application of the frame algorithm as a method for reconstruction of Paley-Wiener functions on weighted graphs from a set of samples (Section 4). In Section 5 these results are illustrated in the case of a line graph.

\section{Proofs of the main results}

\subsection{Proof of Theorem \ref{thm:main}}
We retain the definitions and notations from the previous section.

We introduce a family of auxiliary quantities $\varphi_{m,p}$, for $m \ge 1$, by 
\[
 \varphi_{m,p} = \left( \frac{1}{K_{m-1}} \sum_{u \in S_m} \sum_{v \in S_{m-1}} |f(u)-f(v)|^p w(u,v) \right)^{1/p}~.
\]
Note that by assumption all denominators $K_{m-1}$ are positive. 
%

We then obtain, for $m \ge 1$, 
\begin{equation}
 \label{eqn:f_m_norm}
 \| f_m \|_{p,\nu}\le \left(\frac{D_{m-1}}{K_{m-1}} \right)^{1/p} \| f_{m-1} \|_{p,\nu}+ \varphi_{m,p}
\end{equation}

For the proof of that estimate, we compute 
\begin{eqnarray*}
 \| f_m \|_{p,\nu}& = & \left( \sum_{u \in S_m} |f(u)|^p \nu(u) \right)^{1/p} \\
 & = & \left(  \sum_{u \in S_m}   \sum_{v \in S_{m-1}} |f(u)|^p \frac{w(u,v)\nu(u)}{w_{S_{m-1}}(u)}  \right)^{1/p} \\
& \le & \left( \sum_{u \in S_m} \sum_{v \in S_{m-1}} (|f(v)| +|f(u)-f(v)|)^p \frac{w(u,v)\nu(u)}{w_{S_{m-1}}(u)} \right)^{1/p} \\
& \le &  \left( \sum_{u \in S_m} \sum_{v \in S_{m-1}} |f(v)|^p \frac{w(u,v) \nu(u)}{w_{S_{m-1}}(u)} \right)^{1/p} \\
 & & +  \left( \sum_{u \in S_m} \sum_{v \in S_{m-1}} |f(u)-f(v)|^p \underbrace{\frac{w(u,v) \nu(u)}{w_{S_{m-1}}(u)}}_{\le w(u,v)/K_m} \right)^{1/p}~,
\end{eqnarray*}
using the triangle inequality for weighted $\Lp$-norms. By the definition of $K_{m-1}$, it is clear that the second term is $\le \varphi_{m,p}$. For the first term, we find
\begin{eqnarray*}
  \left( \sum_{u \in S_m} \sum_{v \in S_{m-1}} |f(v)|^p \frac{w(u,v) \nu(u)}{w_{S_{m-1}}(u)} \right)^{1/p} 
& = & \left( \sum_{v \in S_{m-1}} |f(v)|^p \nu(v) \sum_{u \in S_m} \frac{w(u,v) \nu(u)}{w_{S_{m-1}}(u) \nu(v)} \right)^{1/p} \\
& \le &  \left(\frac{D_{m-1}}{K_{m-1}} \right)^{1/p}  \| f_{m-1} \|_{p,\nu}~,
\end{eqnarray*} 
where we used that for $v \in S_{m-1}$,
\[
  \sum_{u \in S_m} \frac{w(u,v) \nu(u)}{w_{S_{m-1}}(u) \nu(v)} \le   \sum_{u \in S_m} \frac{w(u,v) }{K_{m-1} \nu(v)}
= \frac{w_{S_m}(v)}{K_{m-1} \nu(v)} \le \frac{D_{m-1}}{K_{m-1}}~.
\]
Thus (\ref{eqn:f_m_norm}) is proved. 

 As a consequence, we can now estimate $\| f_m \|_{p,\nu}$ in terms of $\| f_0 \|_{p,\nu}$ and the $\varphi_{m,p}$:
\begin{equation} \label{eqn:fmp_est}
 \forall m \ge 0 ~:~ \| f_m \|_{p,\nu}\le \left( \prod_{j=0}^{m-1} \frac{D_j}{K_j} \right)^{1/p} \| f_0 \|_{p,\nu}+ \sum_{j=1}^m \varphi_{j,p} \left( \prod_{i=j}^{m-1} \frac{D_i}{K_i} \right)^{1/p}~.
\end{equation} The proof proceeds by induction over $m$: The case $m=0$ is trivial. The induction step is a straightforward application of (\ref{eqn:f_m_norm}), followed by simplification. 

After these preliminary calculations, we can now estimate $\| f \|_{p,\nu}$: Using that $V(G) = \bigcup_{m}S_m$ disjointly, we obtain from (\ref{eqn:fmp_est}) via the triangle inequality that 
\begin{eqnarray}
\nonumber \lefteqn{ \| f \|_{p,\nu}=  \left( \sum_{m=0}^n \| f _m\|_{p,\nu}^p \right)^{1/p}  \le} \\
 & \le & \left( \sum_{m=0}^n \prod_{j=0}^{m-1} \frac{D_j}{K_j} \| f _0 \|_{p,\nu}^p \right)^{1/p} 
 + \left( \sum_{m=1}^n \left( \sum_{j=1}^m \varphi_{j,p} \left( \prod_{i=j}^{m-1} \frac{D_i}{K_i} \right)^{1/p}\right)^p \right)^{1/p}~.   \label{eqn:fast_fertig}
\end{eqnarray}
Note that the first summand is already as required, and it remains to estimate the second one. 
For this purpose, we introduce $\tilde{\varphi}_{j,p} = K_{j-1}^{1/p} \varphi_{j,p}$, and use repeated applications of H\"older's inequality to deduce for $1 < p < \infty$ that

\begin{eqnarray*}  \lefteqn{ \left( \sum_{m=1}^n \left( \sum_{j=1}^m \varphi_{j,p} \left( \prod_{i=j+1}^{m-1} \frac{D_i}{K_i} \right)^{1/p}\right)^p \right)^{1/p} \le}  \\
 & \le & \left( \sum_{m=1}^n \left( \sum_{j=1}^m \tilde{\varphi}_{j,p}^p \right) \cdot 
\left( \sum_{k=1}^m \frac{1}{K_{k-1}^{q/p}} \left( \prod_{i=k}^{m-1} \frac{D_i}{K_i} \right)^{q/p} \right)^{p/q} \right)^{1/p} \\ &\le &
 \left( \sum_{j=1}^n \tilde{\varphi}_{j,p}^p \cdot \left( \sum_{m=j}^n \left( \sum_{k=1}^m \frac{1}{K_{k-1}^{q/p}} \left( \prod_{i=k}^{m-1} \frac{D_i}{K_i} \right)^{q/p} \right)^{p/q} \right) \right)^{1/p} \\
& \le & \left( \sum_{j=1}^n \tilde{\varphi}_{j,p}^p \right)^{1/p} \cdot \left( \sum_{m=1}^n \left( \sum_{k=1}^m \frac{1}{K_{k-1}^{q/p}} \left( \prod_{i=k}^{m-1} \frac{D_i}{K_i} \right)^{q/p} \right)^{p/q} \right)^{1/p}~,
\end{eqnarray*}
where the last inequality is due to the $1-\infty-$H\"older inequality, using in addition that 
\[
 j \mapsto \sum_{m=j}^n \left( \sum_{k=1}^m \frac{1}{K_{k-1}^{q/p}} \left( \prod_{i=k}^{m-1} \frac{D_i}{K_i} \right)^{q/p} \right)^{p/q} 
\] is increasing.

Plugging in the definition of $\tilde{\varphi}_{j,p}$, we find that
\[
 \sum_{j=1}^n \tilde{\varphi}_{j,p}^p = \sum_{j=1}^n  \sum_{u \in S_j} \sum_{v \in S_{j-1}}  |f(u)-f(v)|^p w(u,v) \le \| \nabla_w f \|_{p}^p~,
\] by disjointness of the $S_j$. This concludes the proof for $1 < p < \infty$. For $p=1$, the required estimate follows by a similar (easier) calculation from H\"older's inequality.

\subsection{Proof of Theorem \ref{thm:main_2}}

The proof ideas are very similar to the ones used in the previous subsection, the chief difference being a change in direction: We first obtain {\em lower} estimates for $\| f_{m+1} \|_{p,\nu}$ in terms of $\| f_m \|_{p,\nu}$, and ultimately, in terms of $\| f_0 \|_{p,\nu}$. Then summation over $m$ yields the desired results. 

For $0 \le m < n$, let 
\[
 \hat{\varphi}_{m,p} = \left( \frac{1}{\hat{K}_{m}} \sum_{u \in S_m} \sum_{v \in S_{m+1}} |f(u)-f(v)|^p w(u,v) \right)^{1/p}~.
\]

Essentially the same proof as for (\ref{eqn:f_m_norm}) yields, for $0 \le m< n$, the inequality 
\begin{equation}
 \label{eqn:f_m_norm_lower}
\| f_m \|_{p,\nu}\le \left( \frac{\hat{D}_m}{\hat{K}_m} \right)^{1/p} \|  f_{m+1} \|_{p,\nu} + \widehat{\varphi}_{m,p}~. 
\end{equation}
Indeed, by analogous calculations we obtain
\begin{eqnarray*} 
\| f_m \|_{p,\nu}
& \le &  \left(  \sum_{v \in S_{m+1}} |f(v)|^p \nu(v) \underbrace{\sum_{u \in S_m} \frac{w(u,v) \nu(u)}{w_{S_{m+1}}(u) \nu(v)}}_{\le \hat{D}_m/\hat{K}_m} \right)^{1/p} \\
 & & +  \left( \sum_{u \in S_m} \sum_{v \in S_{m+1}} |f(u)-f(v)|^p \underbrace{\frac{w(u,v) \nu(u)}{w_{S_{m+1}}(u)}}_{\le w(u,v)/\hat{K}_m} \right)^{1/p}~,
\end{eqnarray*}
which yields (\ref{eqn:f_m_norm_lower}). 

Using this observation, we prove
\begin{equation}
\label{eqn:fmp_est_reverse}
 \left( \prod_{j=0}^{m-1} \frac{\widehat{K}_i}{\widehat{D}_i} \right)^{1/p} \| f_0 \|_{p,\nu}\le \| f_m \|_{p,\nu}+ \sum_{j=0}^{m-1} \left( \prod_{i=j}^{m-1} \frac{\widehat{K}_i}{\widehat{D}_i} \right) ^{1/p} \widehat{\varphi}_{j,p}
\end{equation}
For this purpose we first employ (\ref{eqn:f_m_norm_lower}) to show inductively
\[
 \| f_0 \|_{p,\nu}\le \left( \prod_{j=0}^{m-1} \frac{\widehat{D}_i}{\widehat{K}_i} \right)^{1/p} \| f_m \|_{p,\nu}+ \sum_{j=0}^m \left( \prod_{i=0}^{j-1} \frac{\widehat{D}_i}{\widehat{K}_i} \right)^{1/p} \widehat{\varphi}_{j,p}~,
\] and then divide both sides by $\left( \prod_{j=0}^{m-1} \frac{\widehat{D}_j}{\widehat{K}_j} \right)^{1/p}$. 

We are now ready to prove the estimate: Using disjointness of the $S_m$ and the triangle inequality, we get
\begin{eqnarray}
 \nonumber \lefteqn{\| f \|_{p,\nu}\ge \left( \sum_{m=0}^n \| f_m \|_{p,\nu}^p \right)^{1/p}} \\
\nonumber & \ge & \left( \sum_{m=0}^n \left( \| f_m \|_{p,\nu}+ \sum_{j=0}^{m-1} \left( \prod_{i=j}^{m-1} \frac{\widehat{K}_i}{\widehat{D}_i} \right) ^{1/p} \widehat{\varphi}_{j,p} \right)^p \right)^{1/p} \\ \nonumber & & - 
\left( \sum_{m=1}^n \left( \sum_{j=0}^{m-1} \left( \prod_{i=j}^{m-1} \frac{\widehat{K}_i}{\widehat{D}_i} \right) ^{1/p} \widehat{\varphi}_{j,p} \right)^p\right)^{1/p} \\
\nonumber & \ge & \left( \sum_{m=0}^n \left( \prod_{j=0}^{m-1} \frac{\widehat{K}_i}{\widehat{D}_i} \right) \| f_0 \|_{p,\nu}^p
\right)^{1/p} \\
& - & \left( \sum_{m=1}^n \left( \sum_{j=0}^{m-1} \left( \prod_{i=j}^{m-1} \frac{\widehat{K}_i}{\widehat{D}_i} \right) ^{1/p} \widehat{\varphi}_{j,p} \right)^p\right)^{1/p}~.
\end{eqnarray}
Here the last inequality used (\ref{eqn:fmp_est_reverse}).  We now introduce $\tilde{\widehat{\varphi}}_{j,p} = \widehat{K}_j^{1/p} \widehat{\varphi}_{j,p}$, and repeat the arguments following (\ref{eqn:fast_fertig}) to obtain for $p>1$ that 
\begin{eqnarray*}
 \lefteqn{ \left( \sum_{m=1}^n \left( \sum_{j=0}^{m-1} \left( \prod_{i=j}^{m-1} \frac{\widehat{K}_i}{\widehat{D}_i} \right) ^{1/p} \widehat{\varphi}_{j,p} \right)^p\right)^{1/p}  \le} \\
& \le & \left( \sum_{j=0}^{n-1} \tilde{\widehat{\varphi}}_{j,p}^p \right)^{1/p} \cdot \left( \sum_{m=1}^n \left( \sum_{k=0}^{m-1} \frac{1}{\widehat{K}_{k}^{q/p}} \left( \prod_{i=k}^{m-1} \frac{\widehat{K}_i}{\widehat{D}_i} \right)^{q/p} \right)^{p/q} \right)^{1/p} \\
& \le & \| \nabla_w f \|_{p}\cdot \left( \sum_{m=1}^n \left( \sum_{k=0}^{m-1} \frac{1}{\widehat{K}_{k}^{q/p}} \left( \prod_{i=k}^{m-1} \frac{\widehat{K}_i}{\widehat{D}_i} \right)^{q/p} \right)^{p/q} \right)^{1/p} 
\end{eqnarray*} thus finishing the argument for $1 < p < \infty$.
For $p=1$, we obtain 
\[
 \left( \sum_{m=1}^n \left( \sum_{j=0}^{m-1} \left( \prod_{i=j}^{m-1} \frac{\widehat{K}_i}{\widehat{D}_i} \right)  \widehat{\varphi}_{j,1} \right)\right)
\le \| \nabla_w f \|_1 \cdot \max_{k=0,\ldots,n-1} \frac{1}{\widehat{K}_k} \sum_{m=k}^{n}  \left( \prod_{i=k}^{m-1} \frac{\widehat{K}_i}{\widehat{D}_i} \right) ~.
\]

\section{Poincar\'e inequalities and spectral properties  on  graphs}
 
 In this section we are using definition (\ref{L}) and Definition \ref{PW} from the introduction. In particular, we will work in the unweighted $\ell^2$-setting. 
One can show \cite{Pe2008} that a function $f$ belongs to the space  $PW_{\omega} (L_w)$ if and only if for every positive $t>0$ the following inequality holds
\begin{equation}\label{B}
\|L_{w}^{t}f\|_{2}\leq \omega^{t}\|f\|_{2},\>\>\>t>0.
\end{equation}
The next lemma notes a fundamental relationship between weighted Laplacian and weighted gradient $2$-norm, which will allow to apply our sampling estimates to $PW_{\omega}$:
\begin{lemma}\label{grad-laplace}
For all $f \in \Ltwo(G)$, we have
\[
 2 \| L_w^{1/2} f \|_2^2 = \| \nabla_w f \|_2^2~.
\] 
In particular, for  $f \in PW_{\omega}(L_w)$ we have 
\begin{equation}
\label{Bern}
\|\nabla_w f\|_{2}=\sqrt{2}\|L^{1/2}_wf\|_{2}\leq \sqrt{2\omega}\|f\|_{2}.
\end{equation}
\end{lemma}
\begin{prf}
 Let $\mu(u) = w_{V(G)}(u)$.  Then we obtain 
\begin{eqnarray*}
 \langle f, L_w f \rangle & = &  \sum_{u \in V} f(u) \overline{ \left( \sum_{v \in V(G)} \left( f(u) - f(v) \right) w(u,v) \right)} \\
 & = & \sum_{u \in V(G)} \left( |f(u)|^2 \mu(u) - \sum_{v  \in V(G)} f(u) \overline{f(v)} w(u,v) \right) ~.
\end{eqnarray*}
In the same way 
\begin{eqnarray*}
 \langle f, L_w f \rangle & = & \langle L_w f, f \rangle \\
 & = & \sum_{u \in V(G)} \left( |f(u)|^2 \mu(u) - \sum_{v  \in V(G)} \overline{f(u)} f(v) w(u,v) \right)~.
\end{eqnarray*}
Adding these equations yields
\begin{eqnarray*}
 2 \langle f, L_w f \rangle &  = & 2 \sum_{u \in V(G)} \left( |f(u)|^2 \mu(u) - {\rm Re} \sum_{v  \in V(G)} f(u) \overline{f(v)} w(u,v) \right)  \\
& = & \sum_{ u,v \in V(G)} |f(u)|^2 w(u,v) + |f(v)|^2 w(u,v) - 2  {\rm Re}  f(u) \overline{f(v)} w(u,v) \\
& = & \sum_{u,v  \in V(G)} |f(v)-f(u)|^2 w(u,v) \\
& = & \| \nabla_w f \|_2^2~.
\end{eqnarray*}
Thus the first equality follows by taking the square root of $L_w f$, and (\ref{Bern}) easily follows.
\end{prf}

Suppose that $S_{0}$ is the zero set of a non-zero function $f$ in $PW_{\omega}(L_{w}),\>\>\omega>0$. Then Corollary \ref{col_1}, Lemma \ref{grad-laplace}, and inequality (\ref{B}) imply, for any admissible partition $\mathcal{S}$ with initial set $S_0$, 
\begin{equation}\label{zero set}
 \| f |_{V\setminus S_{0}}\|_2 \le \delta_{ \mathcal{S},{2}}\| \nabla_{w} f \|_2\leq \sqrt{2}   \delta_{ \mathcal{S},{2}}\omega^{1/2} \|  f \|_2=\sqrt{2}   \delta_{ \mathcal{S},{2}}\omega^{1/2} \|  f| _{V\setminus S_{0}}\|_2.
\end{equation}
Thus, one has the inequality 
\begin{equation}\label{Z_1}
 \delta_{ \mathcal{S},{2}}\geq \frac{1}{\sqrt{2\omega}}.
\end{equation}
At the same time the Corollary \ref{col_2} gives for the same  function $f\in PW_{\omega}(L_{w})$
 \begin{equation}
 \| f|_{V\setminus S_{0}}\|_2 \leq \frac   {   \widehat{ \delta}_{ \mathcal{S},{2}}    }{ \widehat{ a}_{ \mathcal{S},{2}}   -1} \| \nabla_{w} f \|_2 \leq \sqrt{2}\frac{ \widehat{ \delta}_{ \mathcal{S},{2}}   }{ \widehat{ a}_{ \mathcal{S},{2}}-1}\omega^{1/2} \|  f \|_2=
 $$
 $$
 \sqrt{2}\frac{ \widehat{ \delta}_{ \mathcal{S},{2}}   }{\widehat{ a}_{ \mathcal{S},{2}}-1}\omega^{1/2} \|  f| _{V\setminus S_{0}}\|_2
 \end{equation}
that implies the inequality
\begin{equation}\label{Z_2}
\frac{ \widehat{ \delta}_{ \mathcal{S},{2}}   }{\widehat{ a}_{ \mathcal{S},{2}}-1}\geq \frac{1}{\sqrt{2\omega}}.
\end{equation}
Hence we have proved the following.
\begin{theorem}\label{zeroset}
Suppose that   $S_{0}$ is zero set of a function $f\in PW_{\omega}(L_{w}),\>\>\omega>0,$  which is not identical zero. Then  for any admissible partition $\mathcal{S}$ with initial set $S_0$ the inequalities (\ref{Z_1}) and (\ref{Z_2}) hold.
\end{theorem}

 We consider a  finite connected graph $G$ of $N$ vertices. Under this assumption the Laplace operator $L_{w}$ has discrete spectrum $0=\lambda_{0}< \lambda_{1}\leq ...\leq \lambda_{N-1}$. Let $u_{0}, u_{1},...u_{N-1}$ be the corresponding orthonormal basis of  $\Ltwo(G)$ where 
 $$
 L_{w}u_{j}=\lambda_{j}u_{j},\>\>\>j=0,...N-1.
 $$
 Let $\mathcal{N}[0,\omega)$ denote the number of eigenvalues of
$L$ in $[0,\omega)$, counted with multiplicities. 
The notation $\mathcal{N}[\omega, 
\lambda_{N-1}]$ is used to denote  a number of eigenvalues of $L$ in
$[\omega, \lambda_{N-1}]$.

In this situation $PW_{\omega} (L_{w})$ is   the span of all eigenfunctions $u_{j}$ for which corresponding eigenvalues $\lambda_{j}$ are not greater than $\omega$.

We are using the same notations as above.

\begin{theorem} 
Suppose that $\mathcal{S}=\{S_{0},...,S_{n(\mathcal{S})}\}$ is an admissible partition of $V(G)$. Then the following inequality holds
\begin{equation}
\mathcal{N}\left [0, \>\>\left(2\delta_{\mathcal{S},2}^{2}\right)^{-1}\right)\leq|S_{0}|,
\end{equation}
and if $|V(G)|=N$, then
\begin{equation}
\mathcal{N}\left [\left(2\delta_{\mathcal{S},2}^{2}\right)^{-1},\>\>\lambda_{N-1}\right]\geq N-|S_{0}|
\end{equation}
\end{theorem}

\begin{proof} According to Theorem  \ref{thm:main} and Lemma \ref{grad-laplace}  if   $\omega<\left(2\delta_{\mathcal{S},2}^{2}\right)^{-1}$ then   $S_{0}$ is a
uniqueness set for the space $PW_{\omega}(G)$. Since dimension of 
$PW_{\omega}(G)$ is exactly  the number of eigenvalues
(counted with multiplicities)  of $L_{w}$ on the interval $[0,\> \omega)$ it means that $|S_{0}|$  cannot be less than the number of eigenvalues
 of $L_{w}$ on the interval $[0,\>\>
\omega)$ for every $\omega<\left(2\delta_{\mathcal{S},2}^{2}\right)^{-1}$. Thus,
$$
\mathcal{N}\left [0, \>\>\left(2\delta_{\mathcal{S},2}^{2}\right)^{-1}\right)\leq|S_{0}|.
$$
The second inequality is obvious.
\end{proof}

Fix a subset $S_{0}\subset V(G)$ and consider all partitions $\mathcal{S}_{S_{0}}=\{S_{0},...,S_{n(\mathcal{S})}\}$ of $V(G)$ for which $K_m(\mathcal{S}) >0$ for all $m=0,\ldots, n(\mathcal{S})-1,$ then the previous Theorem implies the inequalities 
\begin{equation}
\max_{\mathcal{S}_{S_{0}}}\mathcal{N}\left [0, \>\>\left(2\delta_{\mathcal{S}_{S_{0}}, 2}^{2}\right)^{-1}\right)\leq|S_{0}|,
\end{equation}
and 
\begin{equation}
\min_{\mathcal{S}_{S_{0}}}\mathcal{N}\left [\left(2\delta_{\mathcal{S}_{S_{0}}, 2}^{2}\right)^{-1},\>\>\lambda_{N-1}\right]\geq N-|S_{0}|
\end{equation}

In the simplest case $S_{0}\cup bS_{0}=V(G)$ we obtain the following.

\begin{corollary}
For any $S_{0}\subset V(G)$ such that $S_{0}\cup bS_{0}=V(G)$ the following inequality holds  
\begin{equation}
\mathcal{N}\left [0, \>\>K_{0}(S_{0})/2\right)\leq|S_{0}|,
\end{equation}
and if $|V(G)|=N$, then
\begin{equation}
\mathcal{N}\left [K_{0}(S_{0})/2,\>\>\lambda_{N-1}\right]\geq N-|S_{0}|
\end{equation}

\end{corollary}

The following Corollary gives a lower bound for each non-zero
eigenvalue. 

Note that this result is "local" in the sense that any randomly
chosen set  $S_{0} \subset V(G)$ can be used to
obtain an estimate (\ref{from below}) below.

\begin{corollary}
\label{col-II}
Suppose that $|S_{0}|\leq k$ and $\mathcal{S}_{S_{0}}=\{S_{0},...,S_{n(\mathcal{S})}\}$ is an admissible partition of $V(G)$. Then the following inequality holds

\begin{equation}
\label{from below}
\lambda_{k}\geq \max_{ \mathcal{S}_{S_{0}}  }\left(2\delta_{\mathcal{S}_{S_{0}},2}^{2}\right)^{-1}~,
\end{equation}
where $\max$ is taken over all admissible partitions $\mathcal{S}_{S_{0}}$ with $|S_{0}|\leq k$.
\end{corollary}
\begin{proof}

Pick any $S_{0}\subset V(G)$ and let $\mathcal{S}_{S_{0}}=\{S_{0},...,S_{n(\mathcal{S})}\}$ be an admissible partition of $V(G)$. Select any $\lambda_{k}, \>\>0\leq k\leq N-1, \>\>\>N=|V(G)|. $

 If   $\lambda_{k}<\left(2\delta_{\mathcal{S}_{S_{0}},2}^{2}\right)^{-1}$ then   $S_{0}$ is a
uniqueness set for the space $PW_{\lambda_{k}}(G)$. Since dimension of 
$PW_{\lambda_{k}}(G)$ is exactly  $k+1$ it implies that $|S_{0}|\geq k+1$.

From here we obtain that if $|S_{0}|\leq k$ then the inequality  $\lambda_{k}\geq \left(2\delta_{\mathcal{S}_{S_{0}},2}^{2}\right)^{-1}$ holds.

Corollary is proved.
\end{proof}

Given an $\mathcal{M}\subset V(G)$ let us introduce $\Lambda_{D}(\mathcal{M})$ ("a Dirichlet " eigenvalue) as 

$$
\Lambda_{D}(\mathcal{M})=\inf_{f\in \Ltwo(\mathcal{M}),\>\>f\neq 0}\frac{<f, L_{w}f>}{\|f\|_{2}^{2}}.
$$
where $ \Ltwo(\mathcal{M})$ is the set of all functions supported on  $\mathcal{M}$.
\begin{theorem}
For any $S_{0}$ and  any admissible partition $\mathcal{S}_{S_{0}}=\{S_{0},..., S_{n(\mathcal{S})}\}$ with initial set $S_0$
\begin{equation}
\Lambda_{D}\left( V\setminus S_{0}\right)\leq \sqrt{2} \delta_{\mathcal{S}_{S_{0}}, 2},
\end{equation}
and then
\begin{equation}
\Lambda_{D}\left( V\setminus S_{0}\right)\leq \sqrt{2} \min_{ \mathcal{S}_{S_{0}}  } \delta_{\mathcal{S}_{S_{0}}, 2}
\end{equation}
where $\min$ is taken over all  admissible partition $\mathcal{S}_{S_{0}}=\{S_{0},..., S_{n(\mathcal{S})}\}$ with initial set $S_0$.
\end{theorem}
\begin{proof}
Obviously, $\Lambda_{D}(\mathcal{M})$ is the $\textit{smallest}$ constant such that for all functions $f$ supported on $\mathcal{M}$

\begin{equation}
\label{smallest}
\|f\|_{2}\leq \Lambda_{D}(\mathcal{M})\|L_{w}^{1/2}f\|_{2}.
\end{equation}
Take an $S_{0}$ and any  admissible partition $\mathcal{S}_{S_{0}}=\{S_{0},..., S_{n(\mathcal{S})}\}$ with initial set $S_0$. If $f$ is supported on $V\setminus S_{0}$ then we have 
$$
\|f\|_{2}\leq  \sqrt{2}   \delta_{ \mathcal{S}_{S_{0}},{2}} \|L_{w}^{1/2}f\|_{2}
$$
and together with (\ref{smallest}) it gives the inequality
\begin{equation}
\Lambda_{D}(V\setminus S_{0})\leq  \sqrt{2}\min_{ \mathcal{S}_{S_{0}}  } \delta_{\mathcal{S}_{S_{0}}, 2}.
\end{equation}
Theorem is proved.
\end{proof}

If  $\mathcal{M}\cup b\mathcal{M}=V(G)$ then our $\Lambda_{D}(\mathcal{M})$ coincides with the Dirichlet eigenvalue $\lambda_{D}(\mathcal{M})$ introduced in \cite{Ch}.
If $\delta$ is the isoperimetric dimension of a graph $G$ it is known \cite{Ch} that there exists a constant $C_{\delta}$ which depends just on
$\delta$  such that
\begin{equation}
\lambda_{D}(\mathcal{M})>C_{\delta}\left(\frac{1}{vol \>\mathcal{M}}\right)^{2/\delta},
vol \>\mathcal{M}=\sum_{v\in \mathcal{M}}d(v).
\end{equation}
Thus, we obtain 
\begin{theorem} 
If $\delta$ is the isoperimetric dimension of the graph $G$ and $S_{0}\cup bS_{0}=V(G)$ then
there exists a constant $C_{\delta}$ which depends just on
$\delta$  such that the following inequality holds
\begin{equation}
C_{\delta}\left(\frac{1}{vol \>bS_{0}}\right)^{2/\delta}\leq  \delta_{\mathcal{S}_{S_{0}}, 2},\>\>\>\mathcal{S}_{S_{0}}=\{S_{0}, bS_{0}\}.
\end{equation}
\end{theorem}
We consider several concrete examples illustrating that the constants  obtained in our inequalities are close to optimal. 

\begin{example}

Suppose that $G$ is a  star-graph $\{v_{0}, v_{1}, ..., v_{N}\}$ whose
center is $v_{0}$. Let $S_{0}$ be the vertex $\{v_{0}\}$. Then
$K_{0}=1, D_{0}=N$, and since 

\begin{equation}
\label{ineq:n=2_with_gradient}
\|f\|_{2}\leq \left(1+\frac{D_{0}}{K_{0}}\right)^{1/2}\|f_{0}\|_{2}+\frac{1}{K_{0}^{1/2}}\|\nabla f\|_{2},\>\>\>f_{0}=f|_{S},
\end{equation}
we have
$$
\|f\|_{2}\leq \sqrt{N+1}|f(v_{0})|+\|\nabla f\|_{2}.
$$
In particular,  for the constant function  $f(v_{j})=1,\>\>\> 0\leq j\leq N$  one has  $
\|f\|_{2}=\sqrt{N+1},\>\>\> \|\nabla  f\|_{2}=0$ and the inequality  (\ref{ineq:n=2_with_gradient}) becomes
$$
\sqrt{N+1}\leq \sqrt{N+1}.
$$
For the same star-graph  $G,\>\>S_{0}=\{v_{0}\}$  and 
$$
\delta_{\mathcal{S}_{S_{0}},2}=\frac{1}{K_{0}^{1/2}}=1
$$ 
we obtain 
according to Corollary \ref{col-II} $$
\lambda_{1}\geq K_{0}/2=1/2.
$$
But for  a star-graph  and our Laplacian $\lambda_{1}=1$.

\end{example}

%

\begin{example}
Let $C_{N}$ be a cycle of $N$ vertices $\{v_{1},...,v_{N}\}$. Take
another vertex $v_{0}$ and make a graph $C_{N}\cup \{v_{0}\}$ by
connecting $v_{0}$ to each of $v_{1},...,v_{N}$. It is the so-called wheel-graph.

 Let
$\lambda_{k}(N)$ be a
  non-zero eigenvalue of the operator $L$ on the graph $C_{N}$
and let $u_{k}$ be a corresponding orthonormal
eigenfunction. Construct a function $\widetilde{u_{k}}$ on
the graph $C_{N}\cup \{v_{0}\}$  such that
$\widetilde{u}_{k}(v)=u_{k}(v)$ if $v\in C_{N}$ and
$\widetilde{u}_{k}(v_{0})=0$. Since $u_{k}$ is
orthogonal to the constant function $\textbf{1}$ we have that
$$
\sum_{v_{j}\in C_{N}}u_{k}(v_{j})=0
$$
and it implies that for the operator $L$ on $C_{N}\cup \{v_{0}\}$
$$
L\widetilde{u}_{k}(v_{0})=0.
$$
Clearly,  for every $v_{j}\in C_{N}$ one has
$$
L\widetilde{u}(v_{j})=L u_{k}(v_{j})+u(v_{j})=
\left(\lambda_{k}(N)+1)u(v_{j}\right) .
$$ Thus,
$$
L\widetilde{u}_{k}=\left(\lambda_{k}(N)+1\right)\widetilde{u}_{k},
$$
and since $\|\widetilde{u}_{k}\|_{2}=1$ we have that
$$
\|L^{1/2}\widetilde{u}_{k}\|_{2}=\left(\lambda_{k}(N)+1\right)^{1/2}.
$$
Let $S_{0}$ be the graph $C_{N}=\{v_{1},...,v_{N}\}$. In this case the
boundary of $S_{0}$ is the point $v_{0}$,  $K_{0}=N$, $D_{0}=1$
and then for the function $\widetilde{u}_{k}$ the Theorem  \ref{thm:main} implies

\begin{equation}
\label{ineq:n=2}
\|\widetilde{u}_{k}\|_{2}\leq \left(1+\frac{D_{0}}{K_{0}}\right)^{1/2}\|\widetilde{u}_{k}|_{S_{0}}\|_{2}+\frac{1}{K_{0}^{1/2}}\|\nabla \widetilde{u}_{k}\|_{2}=
$$
$$
\left(1+\frac{D_{0}}{K_{0}}\right)^{1/2}\|\widetilde{u}_{k}|_{S_{0}}\|_{2}+\sqrt{\frac{2}{K_{0}}}\|L^{1/2} \widetilde{u}_{k}\|_{2}
\end{equation}
or
$$
1\leq \sqrt{1+\frac{1}{N}}+\sqrt{2}\sqrt{\frac{\lambda_{k}(N)+1}{N}}.
$$
Since all eigenvalues $\lambda_{k}(N)$ belong to $[0,\>4]$, we see that the right-hand side of the last inequality goes to one when $N$ goes to infinity.
\end{example}

\section{Reconstruction of Paley-Wiener functions using frame algorithm}

We are going to apply Corollary \ref{cor:samp_2} to functions in $PW_{\omega}(L_{w})$. 
Thus, we take $ \sqrt{2\omega}$ as the $\epsilon$ in Corollary \ref{cor:samp_2} and we make an assumption  that  the following inequality holds
\begin{equation}
\omega<\frac{1}{2}\left(                \sum_{m=1}^n  \sum_{j=1}^m \frac{1}{K_{j-1}}  \prod_{i=j}^{m-1} \frac{D_i}{K_i}             \right)^{-1}.
\end{equation}
In this case we have non-trivial Plancherel-Polya-type inequalities (\ref{eqn:cor_norm_equiv_sharp}).  Let us denote by  $\theta_{v}$, where $v\in S$, the orthogonal projection of the Dirac measure $\delta_{v},\>\>v\in S$, onto the space $PW_{\omega}(L_{w})$. Since for functions in $PW_{\omega}(L_{w})$ one has $f(v)=\left<f, \theta_{v}\right>,\>\>v\in S$,   the inequality (\ref{eqn:cor_norm_equiv_sharp}) takes the form of a frame inequality in the Hilbert space $H=PW_{\omega}(L_{w})$
\begin{equation}
\label{frame-ineq}
\left( \frac{1-\epsilon \delta_{S,2}}{a_{S,2}} \right)^{2}\| f \|_2^{2} \le \sum_{v\in S}|\left<f,\theta_{v}\right>|^{2}                   \le \left(\frac{1+\epsilon \hat{\delta}_{S,2}}{\hat{a}_{S,2}}\right)^{2} \| f \|_2^{2},\>\>\epsilon=\sqrt{2\omega},
\end{equation}
for all $\>\>f\in PW_{\omega}(L_{w})$. According to the general theory of Hilbert frames \cite{Gr} the last inequality implies that there exists a dual frame (which is not unique in general) $\{\Theta_{v}\},\>v\in S,\>\>\Theta_{v}\in PW_{\omega}(L_{w})$, in the space $PW_{\omega}(L_{w})$ such that for all $f\in PW_{\omega}(L_{w})$ the following reconstruction formula holds 
\begin{equation}
\label{reconstruction}
f=\sum_{v\in S}f(v)\Theta_{v}.
\end{equation}

Suppose that $S_0 \subset V(G)$ is given. We want to determine sufficient conditions for $\omega$ to ensure that $S_0$ is a sampling set for $PW_\omega(L_w)$. For this purpose, we let $S_1 = V(G) \setminus S_0$, and compute the quantities $a_{\mathcal{S},2}$ and $\delta_{\mathcal{S},2}$ for $\mathcal{S} = (S_0,S_1)$. This requires computing
\[
 D_0 = D_{0}(\mathcal{S})  = \sup_{v \in S_0} w_{S_1}(v) 
\]and 
\[
 K_0 = K_0 (\mathcal{S}) = \inf_{v \in S_{1}} w_{S_0(v)}~.
\] 
In order to meet the requirements of Theorem \ref{thm:main}, each $v \in S_1$ must be connected to at least one $v \in S_0$. 
The constants in Corollary \ref{cor:samp_1} are then computed as
$$
a_{\mathcal{S},2}=\left(1+\frac{D_{0}}{K_{0}}\right)^{1/2},\>\>\>
\delta_{\mathcal{S},2}=\frac{1}{K_{0}^{1/2}}.
$$
Thus, we have 
\begin{equation}
\|f\|_{2}\leq \left(1+\frac{D_{0}}{K_{0}}\right)^{1/2}\|f_{0}\|_{2}+\frac{1}{K_{0}^{1/2}}\|\nabla_{w}f\|_{2}=
$$
$$
\left(1+\frac{D_{0}}{K_{0}}\right)^{1/2}\|f_{0}\|_{2}+\sqrt{\frac{2}{K_{0}}}\|L_{w}^{1/2} f\|_{2}.
\end{equation}
In particular, applying (\ref{Bern}) along with assumption
\begin{equation}
\label{condition}
\omega<\frac{K_{0}}{2}
\end{equation}
yields  the following sampling estimate for all $f \in PW_\omega(L_w)$;
\begin{equation}
\label{rightestimate}
\|f\|_{2}\leq \left(1-\sqrt{\frac{2\omega}{K_{0}}  }\right)\left(1+\frac{D_{0}}{K_{0}}\right)^{1/2}\|f_{0}\|_{2},\>\>\>f_{0}=f|_{S_{0}}.
\end{equation}
At the same time

$$
\hat {a}_{S,2}=\left(1+\frac{\hat{K}_{0}}{\hat{D}_{0}}\right)^{1/2},\>\>\>
\hat{\delta}_{S,2}=\frac{1}{\hat{D}^{1/2}_{0}}.
$$
This yields the norm estimate
\begin{equation}
\label{ineq:S_{1}}
\|f\|_{2}+\frac{1}{\hat{D}_{0}^{1/2}}\|\nabla_{w} f\|_{2}\geq \left(1+\frac{\hat{K}_{0}}{\hat{D}_{0}}\right)^{1/2}\|f_{0}\|_{2},\>\>\>f_{0}=f|_{S}.
\end{equation}
If (\ref{Bern}) holds, then
\begin{equation}
\label{leftside}
 \left(1+\frac{\hat{K}_{0}}{\hat{D}_{0}}\right)^{1/2}\|f_{0}\|_{2}\leq \|f\|_{2}+\frac{1}{\hat{D}_{0}^{1/2}}\|\nabla_{w} f\|_{2}\leq \left(1+\sqrt{\frac{2\omega}{\hat{D}_{0}}}\right)\|f\|_{2}.
\end{equation}
After all, for functions $f$ in $PW_{\omega}(L_{w})$ with $\omega <K_{0}/2$ we obtain the following frame inequality

\begin{equation}
A\|f\|^{2}_{2}\leq \sum_{v\in S}|<f,\theta_{v}>|^{2}     \leq B\|f\|_{2}^{2},\>\>\>f_{0}=f|_{S},
\end{equation}
where
\begin{equation}
\label{A}
A=\frac
{    \left(1-\sqrt{\frac {2\omega}{K_{0}}}\right)^{2} }
{        1+\frac{D_{0}}{K_{0}}   }, \>\>\>
B=\frac
{    \left(1+\sqrt{\frac {2\omega}{\widehat{D_{0}}}}\right)^{2} }
{        1+\frac{\widehat{K_{0}}}{\widehat{D_{0}}}  }.
\end{equation}
It shows that if the condition $\omega<K_{0}(S_{0})/2$ is satisfied the set $S_{0}$ is a sampling set for the space $PW_{\omega}(L_{w})$ and  there exists a dual frame (which is not unique in general) $\{\Theta_{v}\},\>v\in S,\>\>\Theta_{v}\in PW_{\omega}(L_{w})$, in the space $PW_{\omega}(L_{w})$ such that for all $f\in PW_{\omega}(L_{w})$ the following reconstruction formula holds 
\begin{equation}
\label{reconstruction2}
f=\sum_{v\in S}f(v)\Theta_{v}.
\end{equation}
However, it is not easy to find a dual frame $\{\Theta_{v}\},\>v\in S$. For this reason we are going to adopt the frame algorithm (see  \cite{Gr}, Ch. 5) for reconstruction of functions in $PW_{\omega}(L_{w})$ from the set $S_{0}$.       

 Let us denote by  $\theta_{v}$, where $v\in S_{0}$, the orthogonal projection of the Dirac measure $\delta_{v},\>\>v\in S_{0}$, onto the space $PW_{\omega}(L_{w})$. 
Given a relaxation parameter $0<\nu<\frac{2}{B}$, where $B$ is defined in (\ref{A})  consider the recurrence sequence $g_{0}=0$, and 
\begin{equation}
\label{rec2}
g_{n}=g_{n-1}+\nu \sum_{v\in S_{0}}(f-g_{n-1})(v)\theta_{v}.
\end{equation}
The reconstruction method for functions  in $ PW_{\omega}(L_{w})$ from their values on $S_{0}$ is the following. 

\begin{theorem}
\label{recon} For  $\mathcal{S} = (S_0,S_1)$,
if  the assumption (\ref{condition}) holds, then for all $f\in PW_{\omega}(L_{w})$ the following inequality holds for all natural $n$
\begin{equation}
\|f-g_{n}\|_2\leq\eta^{n}\|f\|_2,
\end{equation}
where the convergence factor $\eta$ is given by 

\begin{equation}
\label{fac}
\eta=\frac{B-A}{A+B},
\end{equation}
 and $A$ and $B$ are defined in (\ref{A}).
\end{theorem}

\section{Sampling on the integers}
\label{sect:shannon_sampling}

In this section, we want to gauge the precision of our sampling estimates by studying a setup for which the optimal answers are known. We consider the Cayley-graph of the group $\mathbb{Z}$, associated to the symmetric system of generators given by $\{ \pm 1\}$. 
I.e., $u \sim v$ iff $|u-v|=1$. The sampling sets we consider are subgroups of the type $k \mathbb{Z}$, with $k \in \mathbb{Z}$. We want to identify the critical value $\omega_0 = \omega_0(k)$ with the property that for all $\omega< \omega_0$, Corollary \ref{cor:samp_2} is applicable. As will be shortly seen, the spectrum of the associated graph Laplacian is identical to the Fourier spectrum (up to a certain reparameterization), which will allow us to compare the value $\omega_0(k)$ to the optimal value derived from the Shannon sampling theorem. 

The Fourier transform of $f \in \ell^1(\mathbb{Z})$ is defined as 
\[
 \widehat{f}(\xi) = \sum_{k \in \mathbb{Z}} f(k) e^{-ik\xi}~.
\] The Plancherel transform of $f \in \ell^2(\mathbb{Z})$ is denoted by the same symbol. We consider $\widehat{f}$ as a function on the interval $[-\pi,\pi]$. The graph Laplacian associated to the Cayley graph is a convolution product
\[
 L(f) = f \ast C_L~,~ \mbox{ with } C_L = 2 \delta_0 - \delta_1 - \delta_{-1}~.
\]
Hence, by the convolution theorem 
\begin{equation} \label{eqn:laplacian_mult}
 \widehat{L(f)}(\xi) = \widehat{f}(\xi)  \widehat{C}_L(\xi)~,\mbox{ where } \widehat{C}_L(\xi ) = 2 -2\cos(\xi)~. 
\end{equation}
The Fourier-analytic Paley-Wiener space is defined as
\[ 
PW_{\omega} = \{ f \in \ell^2(\mathbb{Z}) : \widehat{f}(\xi) = 0 \mbox{ a.e. outside } [-\omega,\omega] \}~.
\] The Fourier transform $\widehat{C}_L$ is a positive even function that strictly increases on $[0,\pi]$, hence (\ref{eqn:laplacian_mult}), together with the Plancherel transform, implies for $\omega \in [0,\sqrt{2}]$ that 
\[
 PW_{\omega}(L) = PW_{\omega'} ~,~ \omega' = \psi (\omega)~.
\] Here $\psi$ denotes the inverse map of the restriction of $ \omega \mapsto \sqrt{2-2\cos(\xi)}$ to the interval $[0,\pi]$. 

The Shannon sampling theorem provides a sharp characterization of the bandwidth $\omega_0(k)$: 
\begin{theorem}
 $S = k \mathbb{Z}$ is a sampling set for $PW_\omega$ iff $k \omega \le \pi$. In this case, we have the norm equality
\begin{equation} \label{eqn:Shannon}
 \forall f \in PW_\omega~:~ \frac{1}{\sqrt{k}} \| f \|_2 =  \| f|_S \|_2~.
\end{equation}
\end{theorem}

Let us now determine the constants entering the graph-theoretic criteria. For simplicity, we only consider $S = k \mathbb{Z}$, with $k = 2n+1$ an odd integer, and use the partition given by $S_m = b(cl^m(S))$. Then 
\[
 cl^m(S) = \{ \ell + k \mathbb{Z} : |\ell| \le m \}~,
\] in particular, $cl^n(S) = \mathbb{Z}$ and
\[
 S_m = \{ \pm m \} + k \mathbb{Z}~. 
\] 
For $v \in S$, there exist two $u \in bS$ with $u \sim v$, namely $u = v \pm 1$. Thus $d_0(v) = 2$, and $D_0 =2$. By contrast, $k_0(v) = 1$, which implies $K_0 = 1$. 
For $m \ge 1$, one easily verifies $d_m(v) = 1$, and thus $D_m =1$, and finally $K_m = 1$, which results in  
\begin{equation} \label{eqn:graph_const_1}
 \delta_{S,2} = \sqrt{\sum_{m=1}^n \sum_{j=1}^m 1 } = \sqrt{\frac{n(n+1)}{2}}~, a_{S,2} = \left( \sum_{m=0}^n \prod_{j=0}^{m-1} \frac{D_j}{K_j} \right)^{1/2} = \sqrt{2n+1} = \sqrt{k}~.
\end{equation}
For the determination of the constants entering the upper bound, we note that $\hat{k}_0(v) = 2$ for $v \in S$, and thus $\widehat{K}_0 =2$.  For $0 < m < n$ and $v = m' + k \ell \in S_m$ (with
$m' \in \{ \pm m \}, \ell \in \mathbb{Z}$), there exists precisely one $u = {\rm sign}(m') + v \in S_{m+1}$, which shows that $\widehat{K}_m=1$ and uniquely determines $\hat{v}_1(u)$. Finally, it is easily seen that $\widehat{D}_m = 1$. This allows to determine
\begin{equation}\label{eqn:graph_const_2}
 \hat{\delta}_{S,2} = \sqrt{\frac{n(n-1)}{2}}~,~\hat{a}_{S,2} = \sqrt{k}~.
\end{equation}

We can now apply Corollary \ref{cor:samp_2} to the Fourier-analytic Paley-Wiener spaces:
\begin{corollary} Let $k \in \mathbb{Z}$ be odd, and $S = k \mathbb{Z}$.
 As soon as $\frac{k+1}{2}  \sqrt{2-2\cos(\omega)} < 1$, the graph-theoretic sampling estimates guarantee for all  $f \in PW_\omega$
\begin{equation} \label{eqn:int_samp_graph}
 \frac{1- \frac{k+1}{2} \sqrt{2-2\cos(\omega)}}{\sqrt{k}} \| f\|_2 \le \| f|_S \|_2 \le
\frac{1+ \frac{k+1}{2} \sqrt{2-2\cos(\omega)}}{\sqrt{k}} \| f\|_2~.
\end{equation}
For $(k+1)  \sqrt{2-2\cos(\omega)} < 1$, the tightness of the frame estimate is less than or equal to $1 + 2 (k+1) \sqrt{2-2\cos(\omega)}$.  
\end{corollary}
\begin{prf}
It is straightforward to check by the definition of Plancherel's theorem and  $\| \nabla f \|_2 = \sqrt{2} \| L^{1/2} f \|_2$ that Corollary \ref{cor:samp_2} is applicable to $PW_\omega$. The sampling estimate hence follows, since $n (n+1) \le \frac{(k+1)^2}{4}$. The tightness of the estimate is defined as the quotient of upper and lower bound, and it can be estimated by the quantity given in the corollary since, for all $\epsilon < 1/2$ we have $\frac{1+\epsilon}{1-\epsilon} < 1+4 \epsilon$. 
\end{prf}\\

We stress that, while the statement concerning the sampling density can be obtained solely from Theorem \ref{thm:main},  the control over the tightness is a consequence of Theorem \ref{thm:main_2}.
For the comparison of (\ref{eqn:int_samp_graph}) with the equality (\ref{eqn:Shannon}),  we first note that we are in the situation outlined after Corollary \ref{cor:samp_2}:  The constants $a_{S,2}$ and $\hat{a}_{S,2}$ coincide for all sampling sets $S$ considered here. Hence they are just normalization constants, and they coincide precisely with the normalization constants of the optimal sampling result (\ref{eqn:Shannon}). 
In order to compare the bandwidth conditions, we use the estimate $(1+\epsilon)|x| \ge \sqrt{2-2\cos(x)}$, for any $\epsilon>0$, 
to obtain the sufficient condition $(k+1)\omega < \frac{2}{1+\epsilon}$. Thus, asymptotically, {\em the graph-theoretic estimate requires an oversampling by a factor slightly larger than $\frac{\pi}{2}$} by comparison to the optimal rate. Furthermore {\em the tightness of the frame estimate is $\le 1 + \nu$, with $\nu$ inversely proportional to the oversampling rate}. 

In the notation of Section 3, the frame bounds in (\ref{eqn:int_samp_graph}) allow to estimate the parameter $\eta$ describing the speed of convergence in the frame recovery algorithm by
\[
 \eta = \frac{B-A}{B+A} \le \frac{k+1}{2} \sqrt{2-2\cos(\omega)}~.
\] Hence the graph-theoretic estimates allow to directly translate the oversampling rate to the reconstruction rate guaranteed for the frame algorithm.  We note that similar observations pertain for even sampling rates, with slightly worse constants. 

It is instructive to compare the sampling density with the one obtainable from the estimates in \cite{PePe2010}. Using the constants in Theorem \cite[2.1]{PePe2010}, we obtain the sufficient density criterion $\sqrt{3^k-1} \cdot \sqrt{2-2 \cos(\omega)} < 1$, and in the absence of an upper sampling estimate improving the trivial estimate $\| f _0 \|_2 < \| f \|_2$, the tightness of the sampling estimate is given by
$\frac{\sqrt{1+5(3^n-1)/2}}{1-\sqrt{3^k-1} \cdot \sqrt{2-2 \cos(\omega)}}$. But this means that the prescribed sampling rate behaves as $\left| \log(|\omega|) \right|$, rather than $1/|\omega|$. Moreover, the tightness of the estimate grows like $1/|\omega|$, as $\omega \to 0$, independent of the oversampling, with poor control over the convergence rate of the frame algorithm.

\bibliographystyle{amsplain}

\end{document}